\documentclass[oneside,10pt]{article}          
\usepackage[b5paper]{geometry}	    
\usepackage{amsfonts,amsmath,latexsym,amssymb} 
\usepackage{theorem}                
\usepackage{mathrsfs,upref}         
\usepackage{mathptmx}		    
\DeclareMathOperator{\Tr}{Tr}	                            
\usepackage{oam}	            
\usepackage{amsrefs}               	            %

\newtheorem{theorem}{Theorem} 
\newtheorem{proposition}{Proposition}          
               
\newtheorem{corollary}{Corollary}

\theoremstyle{definition}

\begin{document}

\title[Eigenvalue Sums of Magnetic Laplacians]{Eigenvalue Sums of Combinatorial Magnetic Laplacians on Finite Graphs}


\author{John Dever}

\address{John Dever,\\ School of Mathematics\\
Georgia Institute of Technology\\
686 Cherry Street\\
Atlanta, GA 30332-0160 USA\\
\email{jdever6@math.gatech.edu}}


\date{DD.MM.YYYY}                               




\begin{abstract}
We give a construction of a class of magnetic Laplacian operators on finite directed graphs. We study some general combinatorial and algebraic properties of operators in this class before applying the Harrell-Stubbe Averaged Variational Principle to derive several sharp bounds on sums of eigenvalues of such operators. In particular, among other inequalities, we show that if $G$ is a directed graph on $n$ vertices arising from orienting a connected subgraph of $d$-regular loopless graph on $n$ vertices, then if $\Delta_\theta$ is any magnetic Laplacian on $G$, of which the standard combinatorial Laplacian is a special case, and $\lambda_0\leq \lambda_1\leq ...\leq\lambda_{n-1}$ are the eigenvalues of $\Delta_{\theta},$ then for $k\leq \frac{n}{2},$ we have \[\frac{1}{k}\sum_{j=0}^{k-1}\lambda_j \leq d-1.\]
        
\end{abstract}

\maketitle



\section{Introduction}

We will assume $G$ to be a finite directed graph without repeated (directed) edges or self edges (loopless), which we realize as a pair of maps $s,t:\mathscr{E}\rightarrow \mathscr{V}$ of finite sets, such that if $e\in\mathscr{E},$ then $s(e)\neq t(e);$ and if $e_1,e_2\in \mathscr{E}$ with $e_1\neq e_2,$ then $(s(e_1),t(e_1))\neq (s(e_2),t(e_2)).$  The maps $s,t$ are called the  source and target, respectively. The set $\mathscr{E}$ is called the set of directed edges, and the set $\mathscr{V}$ is called the set of vertices. 

If $e\in \mathscr{E}$ is an edge, we think of the vertex $s(e)=se\in \mathscr{V}$ as the ``source" of 
the edge $e$ and  the vertex $t(e)=te\in \mathscr{V}$ as the ``target" of the edge $e.$ If $e\in 
\mathscr{E}$ and $v,w\in \mathscr{V}$ with $se=v,$ $te=w,$ we adopt the notations $vw$ and $v \rightarrow w$ for $e.$  

In this paper we shall primarily study a class of operators defined as follows. Given a map $\theta:
\mathscr{E}\rightarrow [-\pi,\pi]$ such that if $vw, wv \in \mathscr{E}$ then $\theta(vw)=-\theta(wv)$; if $f:\mathscr{V}\rightarrow \mathbb{C}$, for $v\in \mathscr{V}$ let 
\[\Delta_{\theta}(f)(v):= \sum_{e\in \mathscr{E}, te=v} (f(v)-e^{i\theta(e)}f(se))+\sum_{e\in \mathscr{E}, se=v} (f(v)-e^{-i\theta(e)}f(te)).\] Such operators are referred to in the literature as discrete magnetic Laplacians, magnetic combinatorial Laplacians, or discrete magnetic Sch\"{o}dinger operators \cite{hig}. The usual (combinatorial) graph Laplacian corresponds to the choice $\theta(e)=0$ for all $e\in \mathscr{E}.$ The signless Laplacian corresponds to the choice $\theta(e)=\pi$ for all $e\in \mathscr{E}.$ Note that since $e^{i\pi}=e^{-i\pi}=-1,$ it is irrelevant for the choice $\theta=\pi$ whether $\theta(uv)=-\theta(vu)$ for $uv,vu\in\mathscr{E}.$

The operators $\Delta_\theta$ may be connected to electromagnetism, justifying the term ``magnetic graph Laplacian." However, our focus in this paper is primarily mathematical. We study the combinatorial and spectral properties of these operators in the case of finite graphs without symmetry. Moreover, we do not insist that the graph be planar or that the $\theta$ values arise from some ``flux" in a physical model. Our study yields a number of inequalites on the sum of eigenvalues of magnetic Laplacian operators on graphs that apply, in particular, to the classical combinatorial graph Laplacian. 
 
 The outline of this paper is as follows. First, in the following section we define magnetic graph Laplacians. We study some algebraic and combinatorial properties of these operators. Next, in section 3 we recall the averaged variational principle from \cite{Var}. Finally, in section 4 we apply the variational principle to derive bounds for eigenvalue sums of graph magnetic Laplacians. 

\section{Magnetic Graph Laplacians}

We adopt the following notation. If $A$ is a finite set, let $C(A)$ be the set of complex valued functions on $A.$ For $a\in A,$ let $\hat{a}$ be the indicator function at $a,$ that is $\hat{a}(x)$ is $1$ for $x=a$ and $0$ for $x\neq a.$ Then the $(\hat{a})_{a\in A}$ form a basis for $C(A).$ As $C(A)$ is isomorphic to $\mathbb{C}^{|A|},$ where $|A|$ denotes the cardinality of $A,$ it is a Hilbert space with inner product induced by our choice of standard basis $(\hat{a})_{a\in A}.$ As a convention \textit{we take all inner products to be conjugate linear in the first argument}. Also, if $T$ is an operator on a finite dimensional Hilbert space, by $T^*$ we mean its adjoint, that is the operator whose matrix representation in the standard basis is the conjugate transpose of the matrix representing $T$ in the standard basis. For $v$ a vertex, let $d_v$ be the degree of $v\in \mathscr{V}.$ For $f:\mathscr{E}\rightarrow \mathbb{C}$ any complex valued function, we may consider the corresponding multiplication operator $\hat{f}:C(\mathscr{E})\rightarrow C(\mathscr{E})$, defined by its action on basis vectors, $\hat{e}\mapsto f(e)\hat{e}$ for $e\in\mathscr{E}.$ Then the adjoint of $f$ is the multiplication operator, again defined by its values on the basis of $C(\mathscr{E})$, $\hat{e}\mapsto \overline{f(e)}\hat{e}.$ By the above notation we denote it by $\widehat{\bar{f}}.$ Let $\theta:\mathscr{E}\rightarrow [-\pi,\pi]$. Then, as above, we may also consider $\widehat{e^{i\theta}}$ as a multiplication operator $C(\mathscr{E})\rightarrow C(\mathscr{E}).$ Observe that $s,t$ induce operators $\hat{s},\hat{t}:C(\mathscr{E})\rightarrow C(\mathscr{V})$ by extending linearly from the action on basis vectors $\hat{s}(\hat{e}):=\widehat{s(e)}$ and $\hat{t}(\hat{e}):=\widehat{t(e)}$ for $e\in\mathscr{E}.$ The algebraic point of view we take is similar to the one in \cite{sunada}.

Let $\theta:\mathscr{E}\rightarrow [-\pi,\pi]$ such that if $v,w\in\mathscr{V}$ with $vw,wv\in \mathscr{E},$ then $\theta(vw)=-\theta(wv)$. Define a quadratic form $Q_{\theta}$ by $Q_{\theta}(f)=\sum_{e\in\mathscr{E}}|f(te)-e^{i\theta(e)}f(se)|^2.$  Then let $d_{\theta}:=\hat{t}-\hat{s}\widehat{e^{-i\theta}}$ and define the (combinatorial) \textit{magnetic graph Laplacian} \[\Delta_\theta:=d_\theta d_\theta^*=(\hat{t}-\hat{s}\widehat{e^{-i\theta}})
(\hat{t}^*-\widehat{e^{i\theta}}\hat{s}^*).\] Then $Q_\theta(f)=\langle f, \Delta_{\theta} f \rangle.$ Note, by its factorization as a square, $\Delta_{\theta}$ is a positive, self-adjoint operator.

\subsection{Properties of Magnetic graph Laplacians}
Let $D=\hat{s}\hat{s}^*+\hat{t}\hat{t}^*,$ $A_{\theta}:=\hat{s}\widehat{e^{-
i\theta}}\hat{t}^*+\hat{t}\widehat{e^{i\theta}}\hat{s}^*.$ Then by expanding the product $(\hat{t}-\hat{s}\widehat{e^{-i\theta}})(\hat{t}^*-\widehat{e^{i\theta}}\hat{s}^*),$ we see \[\Delta_{\theta}=D-A_{\theta}.\]  Note that $\Delta_0=D-A_0$ is the standard graph Laplacian and that $\Delta_{\pi}=D+A_0.$

It can be seen from either the definition in terms of $s,t$ or from the quadratic form that both the standard Laplacian $\Delta_0$ and $\Delta_{\pi}$ are independent of orientation, as interchanging the roles of $se$ and $te$ for any given edge leaves them invariant. Note, however, that in general $\Delta_\theta$ is highly dependent on the orientation.

Since $\det(\Delta_\theta)=0$ if and only if the lowest eigenvalue $\inf_{\|f\|_2=1} Q_{\theta}f =0,$ and since the set $\|f\|_2=1$ is compact, as the space $C(\mathscr{V})$ is finite dimensional; we have $\det(\Delta_\theta)=0$ if and only if there exists an $f\neq 0$ with $Q_\theta(f)=0.$ This occurs, by the form of $Q_\theta$ given above, if and only if $f(te)=e^{i\theta(e)}f(se)$ for all $e.$ This suggests the following result. 

\begin{proposition} \label{o}
Suppose $k$ is a positive integer. An undirected graph is bipartite if and only if $\det(\Delta_{\frac{\pi}{k}})=0$ for some orientation. An undirected graph is tripartite if and only if $\det(\Delta_{\frac{2\pi}{3}})=0$ for some orientation.
\end{proposition}
\begin{proof}
Without loss of generality we may assume the graph is connected since we may consider its components separately. Let $k$ a positive integer. Let $\omega_1=e^{\frac{i\pi}{k}}.$ Suppose for some orientation that $Q_{\frac{\pi}{k}}(f)=\sum_e|f(te)-\omega_1f(se)|^2=0$ for some $f\neq 0.$ By re-scaling $f$ if necessary and possibly multiplying by a global phase, since $f\neq 0,$ we may assume that $f(v_0)=1$ for some 
$v_0\in\mathscr{V}.$ Then since the graph is connected we have that $f$ takes on at most 
the values $\omega_1^j$ for $j\in\{0,1,...,2k-1\}$. Let $A$ be the set of vertices where $f$ takes on values $\omega_1^{j}$ for $j$ even and $B$ the set of vertices where $f$ takes on values $\omega_1^{j}$ for $j$ odd. Since for any edge $e,$ $f(te)=\omega_1f(se),$ vertices in $A$ can only be connected to vertices in $B$ and vertices in $B$ can only be connected to vertices in $A$. So $A,B$ is a bipartition.
Conversely, suppose $A,B$ is a bipartition of an undirected graph. Then define $f:=1_A+\omega_11_B.$ Define an orientation by having $s$ always take values in $A$, $t$ values in $B$. Then for any $e:se\rightarrow te$ we have $|f(te)-\omega_1f(se)|^2=|\omega_1-\omega_1(1)|^2=0.$ Hence
$Q_{\frac{\pi}{k}}(f)=0.$ 

As for the second assertion, let $\omega_2:=e^{\frac{2i\pi}{3}}$ and suppose for some 
orientation that $Q_{\frac{2\pi}{3}}(f)=\sum_e |f(te)-\omega_2 f(se)|^2=0$ for some 
$f\neq 0.$ Again by re-scaling $f$ if necessary and possibly multiplying by a global phase, since $f\neq 0,$ we may assume that $f(v_0)=1$ for some 
$v_0.$ Then since the graph is connected we have that $f$ takes on at most 
the values $1,\omega_2,\omega_2^2.$ So define 
a tripartition $A_0,A_1,A_2,$ with $A_j,$ for $j=0,1,2,$ the set of vertices where $f$ takes on the value $\omega_2^{j}$. Conversely, suppose $A, B, C$ is a tripartition of an undirected graph. Define $f:=1_A+\omega_21_B+\omega_2^21_C.$ Then define an 
orientation by the following rules. For any edge $e$ between an $A$ vertex 
and a $B$ vertex, take $se$ to be the $A$ vertex, $te$ the $B$ vertex. For any edge $e$ between a $B$ vertex and an $C$ vertex, set $se$ to be the $B$ vertex, $te$ to be the $C$ vertex. For any edge $e$
between a $C$ vertex and an $A$ vertex, set $se$ to be the $C$ vertex, $te$ to be the $A$ vertex.  
Then, by construction, for any directed $e:se \rightarrow te,$ we have $|
f(te)-\omega_2 f(se)|^2=0.$ Hence $Q_{\frac{2\pi}{3}}(f)=0,$ and the proof is complete.
\end{proof}

The above proposition hints at the computational difference between determining whether a graph is 2-colorable or 3-colorable. Indeed, using the proposition to determine whether a graph is 2-colorable requires only computing one determinant since $\Delta_\pi$ is independent of orientation. However, for a graph with $n$ edges, using the proposition to determine whether it is $3$-colorable requires checking at most $2^n$ determinants, one for each orientation.

Following \cite{Lieb}, we call a unitary operator $U:C(\mathscr{V})\rightarrow C(\mathscr{V})$ a gauge transformation if it is multiplication operator with respect to the basis of vertices. So $U\hat{v}=\phi(v)\hat{v}$ for $v\in\mathscr{V}$ where $\phi:\mathscr{V}\rightarrow \mathbb{C}$ with $|\phi|=1.$ 
\begin{proposition} $\Delta_\theta$ is unitarily equivalent under a gauge transformation to the standard Laplacian $\Delta_0$ if and only if $\det(\Delta_\theta)=0.$
\end{proposition}
\begin{proof}
We may assume the underlying graph is connected since $\Delta_\theta$ has a decomposition as a direct sum of corresponding $\Delta_\theta$ operators on each connected component. 

Since $\det(\Delta_0)=0,$ one direction is clear. For the other, suppose 
$\det(\Delta_\theta)=0.$ Let $\phi$, normalized with the supremum norm $\|\phi\|_{\infty}=1,$ such that $\langle \phi,\Delta_\theta \phi \rangle =0.$ Since $\phi$ is normalized and the graph is connected, by the equation \[Q_{\theta}(\phi)=\sum_{e\in \mathscr{E}} |\phi(te)-e^{i\theta(e)}\phi(se)|^2,\] the vanishing of $\langle \phi,\Delta_\theta \phi \rangle$ ensures that $|\phi|=1$. Define $U:C(\mathscr{V})\rightarrow C(\mathscr{V})$ by $U(\hat{v}):=\phi(v)\hat{v}$ for $v\in\mathscr{V}$ and extending linearly. Since $|\phi|= 1,$ $U$ is a gauge transformation. 

Define sesquilinear 
forms $Q,T$ by \[Q(f,g):=\langle f, \Delta_\theta Ug \rangle\;\mbox{and}\; T(f,g):=\langle f, U\Delta_0 g \rangle.\] Let $v,w\in\mathscr{V}.$ We have  
\begin{align*} Q(\hat{v},\hat{w})&=\langle \hat{v}, \Delta_\theta U\hat{w}\rangle = 
\langle \hat{v}, \Delta_\theta \phi(w) \hat{w} \rangle = \phi(w)\langle \hat{v}, 
\Delta_\theta \hat{w}\rangle = \phi(w)(\langle \hat{v}, D\hat{w}\rangle - \langle \hat{v}, A_\theta 
\hat{w} \rangle) \\ &= \phi(w)(\delta_{v,w}\mbox{deg}(v) - \langle \hat{v},(\hat{s}\widehat{e^{-
i\theta}}\hat{t}^*+\hat{t}\widehat{e^{i\theta}}\hat{s}^*) \hat{w} \rangle).
\end{align*} Then if $v=w,$ the result is 
$\phi(v)\deg(v).$ If $v\neq w$ and $v,w$ non-adjacent, then $Q(\hat{v},\hat{w})=0.$ If $e=vw$ is an 
edge, substituting $e^{i\theta(v\rightarrow w)}\phi(v)$ for $\phi(w)$ and $e^{-i\theta(v\rightarrow w)}$ for $\langle \hat{v}, A_\theta \hat{w}\rangle,$ we have $Q(\hat{v},\hat{w})= -\phi(v)$. Similarly an edge of the form $wv$ 
results in $Q(\hat{v},\hat{w})=-\phi(v).$ But also \[T(\hat{v},\hat{w})=\langle \hat{v}, U\Delta_0 \hat{w} \rangle = \langle U^*\hat{v},
\Delta_0 \hat{w} \rangle = \langle \overline{\phi(v)} \hat{v}, \Delta_0 \hat{w} \rangle =\phi(v)\langle \hat{v}, \Delta_0 \hat{w}\rangle,\] which is $\phi(v)$deg$(v)$ if $v=w$, $0$ if $v\neq w$ and $v, w$  non-adjacent, and any edge of the form $vw$ or $wv$ results in $-\phi(v).$ Hence 
$Q=T.$ Therefore $\Delta_\theta U = U\Delta_0,$ which implies $\Delta_\theta$ is unitarily equivalent to $\Delta_0$ under a gauge transformation.
\end{proof}

If $e$ is an oriented edge, say $uv$, let $\bar{e}=vu$ be the reverse edge. Then extend 
$\theta$ by $\theta(\bar{e})=-\theta(e).$ Then if $v=v_0,...,v_m=v$ is a closed 
(unoriented) walk, the flux is defined by $\sum_{j=0}^{m-1} \theta(v_jv_{j+1}) \mbox{\;(mod\;}  2\pi).$ Hence, if $v_jv_{j+1}$ is an oriented edge, then it contributes 
$\theta(v_jv_{j+1})$ to the flux; and if $v_{j+1}v_j$ is an edge, then $v_jv_{j+1} = 
\overline{v_{j+1}v_j}$ contributes $\theta(v_jv_{j+1}) = \theta(\overline{v_{j+1}v_j})=-\theta(v_{j+1}v_j)$ to the flux. 

Note a directed graph has an underlying undirected graph with edge relation $v\sim w$ if $vw$ or $wv$ is a directed edge. By a walk in a graph we mean a finite list of vertices $v_0,v_1,...,v_n$ such that $v_j\sim v_{j+1}$ for $0\leq j<n.$ A closed walk is a walk with the initial and final vertex coinciding. 

The proof of the following proposition, in the case of $A_\theta,$ may be found in \cite{Lieb}. In order to derive this version from the one presented there, note that gauge transformations are diagonal in the standard basis for $C(\mathscr{V})$ and thus commute with $D.$ 

\begin{proposition} Two magnetic Laplacians $\Delta_{\theta_1},$ $\Delta_{\theta_2}$ are unitarily equivalent under a gauge transformation if and only if $\theta_1$ and $\theta_2$ induce the same fluxes through closed walks. \end{proposition}

\section{Averaged Variational Principle}
In this section we develop a tool (see \cite{Var}) for its origin) that will allow estimates on sums of eigenvalues of finite Laplacian operators.

If $M$ is a self-adjoint $n\times n$ matrix, we denote its eigenvalues by 
$\mu_0\leq \mu_1\leq... \leq \mu_{n-1}$ and a corresponding orthonormal 
basis of eigenvectors by $(u_{j})_{j=0}^{n-1}.$ If $V$ is any 
$k$ dimensional subspace of $\mathbb{C}^n$ and $(v_j)_{j=0}^{k-1}$ an orthonormal basis for $V,$ then we define \[\mbox{Tr}(M|_V):=\sum_{i=0}^{k-1} \langle v_i, Mv_i \rangle.\] Tr($M|_V)$ is independent of the basis chosen. Indeed, let $P_V$ be the projection onto $V$. Then $P_V=\sum_{i=0}^{k-1} v_iv_i^*.$ So \[\sum_{j=0}^{n-1} \mu_j \|P_V u_j\|^2 = \sum_{i=0}^{k-1} \sum_{j=0}^{n-1} \mu_j|\langle v_i,u_j\rangle|^2 =\sum_{i=0}^{k-1} v_i^*Mv_i,\] using the spectral decomposition of $M.$ Since the left hand side of the above string of equalities is independent of basis, the result holds.

We begin by stating the following classical result \cite{beck}. 
\begin{proposition} \label{a} With notation as above, for $1\leq k \leq n,$ we have 
\[\sum_{j=0}^{k-1} \mu_j=\inf_{\dim(V)=k} \text{Tr}(M|_V).\] In particular if $(v_i)_{i=0}^{k-1} \subset \mathbb{C}^n$ is any collection of orthonormal vectors, we have \[\sum_{j=0}^{k-1} \mu_j \leq \sum_{j=0}^{k-1} \langle v_j, Mv_j\rangle.\]
\end{proposition}

The following variational principle from \cite{Var} is, as we shall see, a generalization of Proposition \ref{a}. The proof may be found in \cite{Var} and is omitted.
\begin{theorem} (Harrell-Stubbe) \label{c}
Let $M$ be a self adjoint $n\times n$ matrix with eigenvalues $\mu_0\leq \mu_1 \leq ... \leq \mu_{n-1}$ and corresponding normalized eigenvectors $(u_i)_{i=0}^{n-1}.$ Suppose $(Z,\mathscr{M},\mu)$ is a (positive) measure space and $\phi:Z\rightarrow \mathbb{C}^n$ is measurable with $\int_Z \|\phi(z)\|^2 d\mu(z) < \infty$. Then if $Z_0 \in \mathscr{M},$ for any $0\leq k \leq n-1,$ we have 
\[\mu_k \left(\int_{Z_0} \|\phi\|^2 d\mu - \sum_{j=0}^{k-1} \int_Z |\langle \phi,u_j\rangle|^2 d\mu\right) \leq \int_{Z_0} \langle \phi,M\phi\rangle d\mu - \sum_{j=0}^{k-1} \int_Z \mu_j|\langle \phi,u_j\rangle|^2 d\mu.\]
\end{theorem}

Note that provided that $\mu_k \sum_{j=0}^{k-1} \int_Z |\langle \phi,u_j\rangle|^2 d\mu \leq \mu_k \int_{Z_0} \|\phi\|^2 d\mu,$ we have 
that 
\begin{equation} \label{f} \sum_{j=0}^{k-1} \int_Z \mu_j |\langle \phi,u_j\rangle|^2 d\mu \leq \int_{Z_0} \langle \phi,M\phi\rangle d\mu. 
\end{equation}

We now show that Proposition \ref{a} follows from the the previous Theorem \ref{c}, in particular from (\ref{f}). Let $v_1,...,v_{k-1}$ be a collection of orthonormal vectors in $\mathbb{C}^n.$ Take $Z:=\{0,1,...,n-1\}$ and $Z_0:=\{0,1,...,k-1\}$ with the counting measure. Extend the $v_j$ to an orthonormal basis for all of $\mathbb{C}^n.$ Then let $\phi(l):=v_l$. Then \[\sum_{j=0}^{k-1} \int_Z |\langle \phi,u_j\rangle|^2 d\mu = k = \int_{Z_0} \|\phi\|^2 d\mu.\] Hence (\ref{f}) then states 
\[\sum_{j=0}^{k-1}\mu_j \sum_{l=0}^{n-1} |\langle v_l ,u_j \rangle|^2 = \sum _{j=0}^{k-1} \mu_j \leq \sum_{j=0}^{k-1} \langle v_j , Mv_j \rangle,\] and we recover Proposition \ref{a}. 

\section{Inequalities for Sums of Eigenvalues of Magnetic Laplacians}
We now apply the averaged variational principle, Theorem \ref{c}, to $\Delta_\theta.$ \textit{For the remainder of this section we further assume that $G$ is connected, and for any $u,v\in \mathscr{V}$, if $uv\in\mathscr{E}$ then $vu\notin \mathscr{E}.$} In other words, we assume $G$ arises from orienting a connected, loopless, undirected graph without repeated edges. If $G$ has $n$ vertices, let $\lambda_0\leq \lambda_1\leq ... \leq \lambda_{n-1}$ denote the eigenvalues of $\Delta_\theta.$ 

 Let $K_n$ be the complete, loopless, undirected graph on the $n$ vertices of $G$. Orient $K_n$ with some orientation such that the orientation of its restriction to $G$ is the orientation on $G$. Suppose $H$ is a $d$-regular directed subgraph of $K_n,$ with $G$ a directed subgraph of $H.$ This, in particular, implies that $H$ is connected on $n$ vertices. Note this is always possible by taking $H=K_n$ and with $d=n-1$. However, for example, for $G=C_6,$ $d$ may be taken to be $2,3,4,$ or $5.$ Let $H^c$ be the graph complement of $H$ in $K_n$ with the induced orientation. If $e=uv$ is an oriented edge in $K_n$ we shall denote $\bar{e}:=vu$. We call two vertices $u,v$ adjacent and write $u\sim v$ if there is some oriented edge between them, either $u\rightarrow v$ or $v\rightarrow u.$ If the graph $G$ is not clear from the context, we write $\sim_{G}$ if we wish to restrict the relation to $G.$ 
Then let $Z:=\mathscr{V}\times \mathscr{V}$ and $\mathscr{M}:=\mathscr{P}(Z).$ We shall denote pairs $(u,v)$ in $Z$ by $uv$. Let $a,b\geq 0.$ Define $\mu$ on $\mathscr{M}$ by $\mu(e) = \mu(\bar{e})=1$ for $e$ an edge in $H$, $\mu(e)=\mu(\bar{e})=a$ for $e$ an edge in $H^c$, and $\mu(u,u)=b$ for all $u.$ 
Let $\alpha:\mathscr{E}\rightarrow [0,2 
\pi].$ Then, extend $\alpha$ and $\theta$ to all of $K_n$ by setting them equal to $0$ outside of edges of $G.$

Define $\phi_{\alpha,H}:Z\rightarrow C(\mathscr{V})\cong \mathbb{C}^{n}$ by $\phi_{\alpha,H}(uv):=b_{uv},$ where
\begin{displaymath}
   b_{uv} := \left\{
     \begin{array}{lr}
       \hat{v}-e^{i\alpha(uv)}\hat{u}&,\; uv \in \mathscr{E}_{K_n}\\
       \hat{v}+e^{-i\alpha(vu)}\hat{u}&,\; vu \in \mathscr{E}_{K_n}\\
       \hat{u}&,\;\;\;\;\;\;  u=v. 
     \end{array}
   \right.
\end{displaymath} 
Hence for any $f,$ 
\begin{displaymath}
   |\langle f, b_{uv} \rangle|^2 = \left\{
     \begin{array}{lr}
       |f(u)|^2+|f(v)|^2 - 2Re(e^{i\alpha(uv)}\overline{f(u)}{f(v)}) & ,\; uv \in \mathscr{E}_{K_n}\\
       |f(u)|^2+|f(v)|^2 + 2Re(e^{-i\alpha(vu)}\overline{f(u)}{f(v)}) & ,\; vu \in \mathscr{E}_{K_n}\\
       |f(u)|^2 \;& ,\;\;\;\;\;\;  u=v. 
     \end{array}
   \right.
\end{displaymath} 

We wish to calculate $\sum_{uv\in Z} \mu(uv)|\langle f, b_{uv} \rangle|^2.$ Note that for $uv$ an edge in $K_n$, we have $|\langle f, b_{uv} \rangle|^2+|\langle f, b_{vu} \rangle|^2=2|f(u)|^2+2|f(v)|^2.$ For $u$ fixed and for any vertex $v,$ exactly one of the three following possibilities occurs: $v=u$, $v$ is adjacent to $u$ in $H$, or $v$ is adjacent to $u$ in $H^c$. 
Hence, since edges and their opposites occur in pairs in both $H$ and $H^c$, and since $H$ is $d$-regular and $H^c$ is $n-1-d$ regular, we have 
\begin{align*}
&\sum_{uv\in Z} \mu(uv)|\langle f, b_{uv} \rangle|^2 \\&= \sum_u \sum_{v\sim_H u}  (|f(u)|^2+|f(v)|^2) + a(\sum_u \sum_{v\sim_{H^c} u}  (|f(u)|^2+|f(v)|^2))\\&= (d+a(n-1-d)+b)\|f\|^2+d\|f\|^2+(n-1-d)a\|f\|^2 \\&=2(d+a(n-1-d)+\frac{b}{2})\|f\|^2.
\end{align*} Let $C(a,b,d):=d+a(n-1-d)+\frac{b}{2}.$ Then \[\sum_{uv\in Z} \mu(uv)|\langle f, b_{uv} \rangle|^2 = 2C(a,b,d)\|f\|^2.\]

Let $Z_0\subset Z.$  
Now we calculate $\sum_{uv \in Z_0} \langle b_{uv},\Delta_{\theta}b_{uv}\rangle.$ There are four cases. If $uv$ is an oriented edge in $G,$ then
\[\langle b_{uv},\Delta_{\theta}b_{uv}\rangle = |1+e^{i(\alpha(uv)+\theta(uv))}|^2+d_u-1 +d_v-1 = d_u+d_v+2\cos(\alpha(uv)+\theta(uv)).\] If $vu$ is an oriented edge in $G,$ then \[\langle b_{uv},\Delta_{\theta}b_{uv}\rangle = |1-e^{-i(\alpha(vu)+\theta(vu))}|^2+d_u-1 +d_v-1 = d_u+d_v-2\cos(\alpha(vu)+\theta(vu)).\] If $u\neq v$ and neither $uv$ nor $vu$ is an oriented edge in $G,$ then \[\langle b_{uv},\Delta_{\theta}b_{uv}\rangle = d_u+d_v.\] Lastly, for any $v$, \[\langle b_{vv},\Delta_{\theta}b_{vv}\rangle = d_v.\] 
Hence \begin{equation} \label{m}
\begin{split}
&\sum_{uv \in Z_0} \langle b_{uv},\Delta_{\theta}b_{uv}\rangle \mu(uv) = \sum_{\{uv\in Z_0\;|\; uv \in \mathscr{E}\}} d_u+d_v +2\cos(\alpha(uv)+\theta(uv)) \\&+
\sum_{\{uv\in Z_0\;|\; vu \in \mathscr{E}\}} d_u+d_v  -2\cos(\alpha(vu)+\theta(vu))\\&+
\sum_{\{uv\in Z_0\;|\; uv,vu \notin \mathscr{E},u\neq v, uv\; \mbox{or}\; vu \in 
\mathscr{E}_H\}} d_u+d_v \\ &+ 
a(\sum_{\{uv\in Z_0\;|\; uv,vu \notin \mathscr{E},u\neq v, uv,vu \notin \mathscr{E}_H\}} 
d_u+d_v) + b(\sum_{\{v\;|\;vv\in Z_0\}} d_v).
\end{split}
\end{equation}

For $A$ a finite set, let $|A|$ denote the cardinality of $A.$ Then we have 
\begin{equation} \label{n}
\begin{split} \sum_{uv\in Z_0} \|b_{uv}\|^2\mu(uv)& = 2|\{uv\in Z_0\;|\;uv \;\mbox{or} \;vu\in \mathscr{E}_H\}| \\ &+ 2a|\{uv\in Z_0\;|\;uv\; \mbox{or}\; vu \in \mathscr{E}_{H^c}\}| + b|\{u\;|\;uu\in Z_0\}|.
\end{split}
\end{equation}
Hence by (\ref{f}) following Theorem \ref{c}, if $k$ is such that $2kC(a,b,d)\leq \sum_{uv\in Z_0} \|b_{uv}\|^2\mu(uv),$ then \[\sum_{j=0}^{k-1} \lambda_j \leq \frac{\sum_{uv \in Z_0} \langle b_{uv},\Delta_{\theta}b_{uv}\rangle \mu(uv)}{2C(a,b,d)}.\] 

We may achieve great simplifications of the above inequality if we take $Z_0$ to contain only edges or reverse edges of $\mathscr{E},$ or also, if needed, ``loops" of the form $uu$.

Before continuing, we note the following. The quantity $Z_G:=\sum_v d_v^2$ is known in graph theory literature as the first Zagreb index of $G$ (see \cite{zag}), where $d_v$ is the degree of $v$ in $G$. Then note that
\[\sum_{uv\in\mathscr{E}} (d_u+d_v) = Z_G.\] Indeed, for each $v,$ $d_v$ 
appears once in exactly $d_v$ terms in the sum. 

For what follows, let $Z_0=\mathscr{E}.$ 
Then (\ref{m}) simplifies to \[\sum_{uv \in Z_0} \langle b_{uv},\Delta_{\theta}b_{uv}\rangle=Z_G+2\sum_{e\in \mathscr{E}} \cos(\alpha(e)+\theta(e)).\] Since we will be wishing to minimize this quantity, we define $\alpha$ such that \[\alpha(e)+\theta(e) \equiv \pi \mbox{\;(mod\;}  2\pi).\] Then \[\sum_{uv \in Z_0} \langle b_{uv},\Delta_{\theta}b_{uv}\rangle=Z_G-2|\mathscr{E}|.\] 

Applying Theorem \ref{c} and using that the right hand side of equation (\ref{n}) simplifies to $2|\mathscr{E}|$, we have that if $2kC(a,b,d)\leq 2|\mathscr{E}|,$ then \[2C(a,b,d)\sum_{j=0}^{k-1}\lambda_j\leq Z_G-2|\mathscr{E}|.\] However, since $C(a,b,d)$ can take on any number greater than or equal to $d,$ if $k\leq \frac{|\mathscr{E}|}{d},$ then the optimal choice is $C(a,b,d)=\frac{|\mathscr{E}|}{k}.$ Hence, we have proven the following theorem.

\begin{theorem} Suppose $G$ is a directed graph arising from orienting a connected, loopless, undirected graph without repeated edges. Let $d_0$ be the degree  of a regular subgraph $H$ of $K_n$ containing $G$ as a subgraph. Then if $k$ is an integer with $k \leq \frac{|\mathscr{E}|}{d_{0}},$ we have \[\frac{1}{k}\sum_{j=0}^{k-1}\lambda_j \leq \frac{Z_G}{2|\mathscr{E}|}-1.\]
\end{theorem}

Note that if $D$ is the degree matrix, $|\mathscr{E}|=\frac{1}{2}\Tr(D)$ and $Z_{G}=\Tr(D^2).$ Hence we may rewrite the above inequality as follows.
For $G$ as in the previous theorem, we have \[\frac{1}{k}\sum_{j=0}^{k-1} \lambda_j \leq \frac{\Tr(D^2)}{\Tr(D)}-1,\;\mbox{for}\;k\;\mbox{a positive integer with}\; k\leq \frac{1}{2d_0}\Tr(D). \]

We may increase the bound on $k$ by admitting a combination of reverse edges of $G,$ and loops $uu$ to $Z_0.$ Then the cosine terms cancel in pairs for reverse edges and loops add terms proportional to the degree.

In \cite{Lieb} the half-filled band, corresponding to the case that $k=\lfloor{\frac{n}{2}}\rfloor,$ is studied. As a corollary we provide an inequality for the half-filled band in the case of a $d-$regular 
graph. Let $H=G$. Then $d_{0}=d.$ Note in this case $Z_G=nd^2$ and $2|\mathscr{E}|=nd.$ Note further that any $\Delta_\theta$ is a sum of magnetic Laplacians corresponding to individual edges, each being a positive operator. It follows that eigenvalue sums for a subgraph are bounded above by corresponding sums for the graph. Therefore we have the following result for the half-band.

\begin{corollary} Suppose $G$ is a directed graph on $n$ vertices arising from orienting a connected, undirected subgraph of a $d$-regular undirected loopless graph on $n$ 
vertices without repeated edges. Then for $k\leq \frac{n}{2},$ we have \[\frac{1}{k}\sum_{j=0}^{k-1}\lambda_j \leq d-1.\] 
\end{corollary}

Note the above bounds hold for all choices of $\theta.$ In particular they hold for the standard combinatorial Laplacian. This connects to the ''flux phase" problem investigated for the case of planar graphs in \cite{Lieb}, that is to find the choice of $\theta$ that maximizes the sum of the half-band eigenvalues. For different classes of graphs, the optimal choice of $\theta$ may vary, leading to the possibility for improvements to the above bounds in such cases for particular choices of $\theta.$ 

We give two simple examples. Let $G=K_3$ with some orientation. The condition is $k\leq \frac{3}{2},$ so the only non-trivial choice for $k$ is $1$. The above inequality reduces to $\lambda_0 \leq 1.$ This is sharp since taking $\theta$ constant equal to $\pi$ on any orientation yields a spectrum of $1,1,4.$

Consider the cycle $C_4.$ Then the spectrum of the standard Laplacian is $0,2,2,4.$ Hence the inequality is sharp at $k=2$ for this example, as the sum of the first half of the spectrum is $2=\frac{n}{2}(d-1),$ where $n=4,$ $d=2.$

\section{Acknowledgments}
I would like to thank Evans Harrell for helpful discussion of the ideas in this paper and comments on a draft. Also, I would like to thank the anonymous reviewer for comments that led to substantial improvements to an earlier version of this paper.

\Refs

\bibitem{beck}
        \by E. Beckenbach and R. Bellman
        \book Inequalities
        \publ Springer 
        \pages 77
        \yr 1965
\endref

\bibitem{Var}
		\by E. M. Harrell II and J. Stubbe
		\paper On sums of graph eigenvalues
		\jour Linear Algebra and its Applications
		\vol 455
		\issue \phantom
		\pages 168--186
		\yr 2014
		
\endref

\bibitem{hig}
		\by Y. Higuchi and T. Shirai
		\paper A remark on the spectrum of magnetic {L}aplacian on a graph
		\jour Yokohama Math Journal
		\vol 47
		\issue \phantom
		\yr 1999
\endref

\bibitem{Lieb}
\by E. H. Lieb and M. Loss
\paper Fluxes, Laplacians, and Kasteleyn’s theorem
\jour Duke Mathematical Journal
\vol 8
\issue 2
\yr 1993
\pages 337--363
\endref

\bibitem{zag}
\by S. Nikoli{\'c} and G. Kova{\v{c}}evi{\'c} and A. Mili{\v{c}}evi{\'c} and N. Trinajsti{\'c}
\paper The {Z}agreb indices 30 years after
\jour Croatica Chemica Acta
\vol 76
\issue 2
\yr 2003
\pages 113--124
\endref

\bibitem{sunada}
\by T. Sunada
\paper Discrete geometric analysis
\book Geometry on Graphs and its Applications, Proceedings of Symposia in Pure Mathematics
\publ American Mathematical Society
\vol 77
\year 2008
\pages 51--86
\endref
\endRefs

\end{document}